\documentclass[fontsize=11pt]{amsart}

\newcounter{defcounter}
\numberwithin{defcounter}{section}
\numberwithin{equation}{section}

\usepackage{appendix}
\usepackage{amsmath}
\usepackage{amsthm}
\usepackage{amssymb}
\usepackage{calligra}
\usepackage{mathtools}
\usepackage{mathrsfs}
\usepackage{physics}
\usepackage[margin=1in]{geometry}
\usepackage{scrextend}
\usepackage{tikz-cd}
\usepackage[shortlabels]{enumitem}
\usepackage{colonequals}
\usepackage{hyperref}
\usepackage[all]{xy}
\usepackage{pdfsync}
\usepackage[mathscr]{euscript}
\usepackage[notcolon, notquote]{hanging}
\usepackage{verbatim}
\usepackage{tikz}
\usepackage{stackengine}
\usepackage{adjustbox}
\usepackage[nopostdot,nogroupskip,sort=use]{glossaries}

\newcommand{\lra}{\longrightarrow}

\makeatletter
\renewenvironment{proof}[1][\proofname]{\par
  \pushQED{\qed}%
  \normalfont
  \topsep-5\p@\@plus6\p@\relax 
  \trivlist
  \item[\hskip\labelsep
        \upshape
    \emph{#1}\@addpunct{.}]\ignorespaces
}{%
  \popQED\endtrivlist\@endpefalse
}
\makeatother

\theoremstyle{plain}
\newtheoremstyle{mytheoremstyle}{5pt}{-5pt}{\itshape}{}{\bfseries}{.}{.5em}{} 
\theoremstyle{mytheoremstyle}
\newtheorem{theorem}{Theorem}[section]

\theoremstyle{definition}
\newtheoremstyle{mydefinitionstyle}{5pt}{-5pt}{}{}{\bfseries}{.}{.5em}{}
\theoremstyle{mydefinitionstyle}

\newtheorem{corollary}[theorem]{Corollary}

\newtheorem{lemma}[theorem]{Lemma}

\tikzset{
  symbol/.style={
    draw=none,
    every to/.append style={
      edge node={node [sloped, allow upside down, auto=false]{$#1$}}}
  }
}
\tikzcdset{scale cd/.style={every label/.append style={scale=#1},
    cells={nodes={scale=#1}}}}

\author{Justin Lacini}
\address{Department of Mathematics, Texas A$\&$M University, College Station, Texas  77840}
\email{jlacini@tamu.edu}

\begin{document}
\title{Effective syzygies of secant varieties}

\begin{abstract}
We provide effective bounds on adjoint line bundles on $X$ for which the secant variety $\Sigma(X)$ satisfies property $N_{3,p}$. The given bounds only depend on $\operatorname{dim}(X)$ and $p$.
\end{abstract}

\maketitle

\section{Introduction}

Let $X\subseteq \mathbb{P}^r$ be a non-degenerate smooth complex projective variety of dimension $n$. The secant variety 
$\Sigma(X)$ is defined to be 
\[
\Sigma(X):=\overline{\bigcup_{\substack{p,q\in X\\ p\neq q}} \operatorname{Span}\{p,q\}}.
\]

\noindent The purpose of this paper is to prove the following result on the syzygies of $\Sigma(X)$.

\begin{theorem}\label{thm1-intro}
    Let $X$ be a smooth projective variety of dimension $n$, let $A$ be a very ample line bundle, let $B$ be a nef line bundle, and set $L:=\omega_X \otimes A^m\otimes B$. Consider the embedding $X\subseteq \mathbb{P}^r$ given by $L$. If $m\geq 2(2n+p)(n+1)+n^2$ then $\Sigma(X)$ satisfies property $N_{3,p}$. 
\end{theorem}

 Ullery \cite{Ullery.16} showed that $\Sigma(X)$ is normal for $m\geq 2n+2$. Chou and Song \cite{Chou.Song.18} proved that in this range $\Sigma(X)$ has in fact Du Bois singularities. Later it was shown in \cite{Choi.Lacini.Park.Sheridan.25} that for $m\geq 4n$ the embedding $\Sigma(X)\subseteq \mathbb{P}^r$ is projectively normal. After that, Choi and Park \cite{Choi.Park.26} answered \cite[Question 7.11]{Choi.Lacini.Park.Sheridan.25}, showing that the optimal $m\geq 2n+2$ suffices. Moreover, they proved that for $m\geq 3n+2$ the homogeneous ideal of $\Sigma(X)$ is generated by cubics.

Once the degree of the generators of the homogeneous ideal of $\Sigma(X)$ is understood, it is natural to study the degree of their syzygies. Recall that a projectively normal embedding $Y\subseteq \mathbb{P}^r$ satisfies property $N_{k,p}$ if the Koszul cohomology groups $K_{i,j}(Y,\mathcal{O}_Y(1))$ vanish in the range $0\leq i\leq p$ and $j\geq k$. A landmark result of Ein and Lazarsfeld \cite{Ein.Lazarsfeld.93}, later improved in \cite{Bangere.Lacini.25}, shows that $X\subseteq \mathbb{P}^r$ satisfies property $N_{2,p}$ if $\mathcal{O}_X(1)\cong\omega_X\otimes A^m \otimes B$ for $m\geq n+1+p$. Theorem B of \cite{Choi.Lacini.Park.Sheridan.25} shows that if $\mathcal{O}_X(1)$ is ``sufficiently ample" relative to $X$ and $p$, then $\Sigma(X)$ satisfies property $N_{3,p}$. This result is, however, far from effective, and the main goal of the present paper is to provide effective bounds. In particular, Theorem \ref{thm1-intro} combined with \cite{Choi.Park.26} completely answers \cite[Question 7.8]{Choi.Lacini.Park.Sheridan.25} for $\Sigma(X)$ when $A$ is very ample.

The approach we take for Theorem \ref{thm1-intro} is cohomological, within the general framework set in \cite{Choi.Lacini.Park.Sheridan.25}. In particular, proving property $N_{3,p}$ essentially amounts to proving the vanishing
\[\tag{*}
H^i\left(X^{[2]}, \bigwedge^j M_{2,L} \otimes A_{2,L}\right)=0
\]

\noindent in the range $i\geq 1$ and $0\leq j\leq i+p$, where $M_{2,L}$ and $A_{2,L}$ are, respectively, a naturally defined vector bundle and line bundle on the Hilbert scheme of points $X^{[2]}$ (see Section \ref{section-preliminaries} for the details). When studying syzygies on $X$ itself, rather than $\Sigma(X)$, one often needs to provide a vanishing analogous to (*), with $M_{2,L}$ replaced by the syzygy bundle $M_L$ on $X$. In this case, it is standard to reduce the desired vanishing for $\bigwedge^j M_L$ to a vanishing for $M_L ^{\otimes j}$, which is much easier to handle (see for example \cite{Ein.Lazarsfeld.93} and \cite{Bangere.Lacini.25}). 

For the vanishing (*) on the Hilbert scheme $X^{[2]}$, however, the situation is quite different. It was indeed shown in \cite[Remark 4.15]{Choi.Lacini.Park.Sheridan.25} that it may happen
\[
H^1(X^{[2]}, M_{2,L}^{\otimes j}\otimes A_{2,L})\neq 0
\]

\noindent regardless of the positivity of $L$. The key issue is that $M_{2,L}$ acquires negativity when restricting to the exceptional divisor $E$ of the Hilbert--Chow morphism 
\[
h:X^{[2]}\lra \operatorname{Sym}^2 X,
\]

\noindent and this negativity increases under tensor products. This pattern is replicated in the nested Hilbert scheme $X^{[1,2]}$, which may be expressed as the blow up of $X\times X$ along the diagonal $\Delta_X$, with exceptional divisor $F$.
Note that we may compare $X^{[1,2]}$ with $X^{[2]}$ via the finite flat forgetful map $\rho:X^{[1,2]}\lra X^{[2]}$.

Here we overcome the loss of positivity by directly working with the restriction of $\rho^* \bigwedge^j M_{2,L}$ to $F$. The advantage is that $F$ admits the simple description $F=\mathbb{P}(\Omega_X^1)$.
By appropriately filtering $\rho^* \left(\bigwedge^j M_{2,L}\otimes A_{2,L}\right)|_F$, we are then able to obtain vanishing of its higher cohomology via a sequence of vector bundle vanishing theorems on $X$. Moreover, the vanishing holds up to twists by multiples of $F$. This in turn allows us to subtract multiples of $E$ in the original desired vanishing (*), thus increasing the positivity of $\bigwedge^j M_{2,L}\otimes A_{2,L}$. It is at this point that we may replace it with $M_{2,L}^{\otimes j}$, and a separate vanishing theorem for $M_{2,L}^{\otimes j}\otimes A_{2,L}(-aE)$ then concludes the proof. 

\textbf{Acknowledgements.} ChatGPT 6, provided by OpenAI through the Academic Researcher Program, was used to check for correctness. All the proofs have been obtained by the author, who wrote the paper in its entirety.

\section{Preliminaries}\label{section-preliminaries}

Here we briefly recall the set-up of \cite{Choi.Lacini.Park.Sheridan.25}.

\subsection{Hilbert scheme of points}\label{subsection-hilbert-scheme} For $k\geq 1$ we denote by $X^{[k]}$ the Hilbert scheme of $0$-dimensional subschemes of $X$ of length $k$. 
We refer to $0$-dimensional subschemes of length $k$ as $k$-schemes. 
The Hilbert--Chow morphism is the canonical morphism
\[
h:X^{[k]}\lra \operatorname{Sym}^k (X). 
\]

\noindent Whenever $X^{[k]}$ is irreducible the Hilbert--Chow morphism is birational. 
For any line bundle $L$ on $X$, we have an induced line bundle $L^{\boxtimes k}$ on $X^k$, which descends
to a line bundle $S_{k,L}$ on $\operatorname{Sym}^k(X)$. We set $T_{k,L}:=h^* S_{k,L}$.

For $1\leq m\leq k$, we denote by $X^{[m,k]}$ the nested Hilbert scheme of $m$-subschemes of $k$-schemes. In this paper, we will be mostly interested in $X^{[2]}$ and $X^{[1,2]}$. We have natural forgetful morphisms
\[
\tau:X^{[m,k]}\lra X^{[m]}
\]

\noindent and 
\[
\rho:X^{[m,k]}\lra X^{[k]}. 
\]

\noindent There is also a residue map
\[
\operatorname{res}:X^{[k-1,k]}\lra X
\]

\noindent sending the pair $(\xi\subseteq \eta)$ to the point defined by $(\mathcal{I}_\eta:\mathcal{I}_\xi)$. 
Let $\mathcal{Z}_k\subseteq X\times X^{[k]}$ be the universal family. Note that $\operatorname{pr}_2:\mathcal{Z}_k\lra X^{[k]}$ is a finite flat map of degree $k$. Let $L$ be a line bundle and consider the short exact sequence
\[
0\lra \operatorname{pr}_1^* L \otimes \mathcal{I}_{\mathcal{Z}_k}\lra \operatorname{pr}_1^* L \lra \operatorname{pr}_1^* L \otimes \mathcal{O}_{\mathcal{Z}_k}\lra 0. 
\]

\noindent If $L$ separates $k$ schemes, we have an induced short exact sequence
\[
0\lra \operatorname{pr}_{2,*}(\operatorname{pr}_1^* L \otimes \mathcal{I}_{\mathcal{Z}_k})\lra \operatorname{pr}_{2,*} \operatorname{pr}_1^* L \lra \operatorname{pr}_{2,*}(\operatorname{pr}_1 ^* L \otimes\mathcal{O}_{\mathcal{Z}_k})\lra 0. 
\]

\noindent Setting $M_{k,L}:=\operatorname{pr}_{2,*}(\operatorname{pr}_1^* L \otimes \mathcal{I}_{\mathcal{Z}_k})$
and $E_{k,L}:=\operatorname{pr}_{2,*}(\operatorname{pr}_1 ^* L \otimes\mathcal{O}_{\mathcal{Z}_k})$, we may rewrite the above sequence as
\[
0\lra M_{k,L}\lra H^0(X,L)\otimes \mathcal{O}_{X^{[k]}}\lra E_{k,L}\lra 0,
\]

\noindent where all the sheaves are locally free. Note that $E_{k, \mathcal{O}_X}=\operatorname{pr}_{2,*} \mathcal{O}_{\mathcal{Z}_k}$ is a locally free sheaf of rank $k$. We set
\[
\delta_k :=\operatorname{det}(E_{k,\mathcal{O}_X})^{-1}. 
\]

\noindent We also set $N_{k,L}:=\operatorname{det}(E_{k,L})$. We have $N_{k,L}=T_{k,L}\otimes \delta_k ^{-1}$ by \cite[Lemma 2.3]{Choi.Lacini.Park.Sheridan.25}.
Set $A_{k,L}:=T_{k,L}\otimes \delta_k ^{-2}$. 

If $k\leq 3$, the map
\[
\operatorname{res}\times \tau:X^{[k-1,k]}\lra X\times X^{[k-1]}
\]

\noindent may be identified the blowing up along $\mathcal{Z}_{k-1}$. Therefore, we often denote
$\operatorname{res}\times \tau$ as
\[
\operatorname{bl}:X^{[k-1,k]}\lra X\times X^{[k-1]},
\]

\noindent and we set $F_{k-1}:=\operatorname{Ex}(\operatorname{bl})$. We have
\[
\rho^*\delta_k \otimes \tau^* \delta_{k-1}^{-1}\cong \mathcal{O}_{X^{[k-1, k]}}(F_{k-1})
\]

\noindent by \cite[Lemma 2.4]{Choi.Lacini.Park.Sheridan.25}. 

\subsection{Secant varieties} Let $X\subseteq \mathbb{P}^r$ be a smooth projective variety. As discussed in the introduction, the secant variety $\Sigma(X)$ of $X$ is defined as
\[
\Sigma(X):=\overline{\bigcup_{\substack{p,q\in X\\ p\neq q}} \operatorname{Span}\{p,q\}}.
\]

\noindent Set $L:=\mathcal{O}_X(1)$ and suppose that the embedding $X\subseteq \mathbb{P}^r$ is linearly normal. Then we may construct a morphism, which is a resolution of singularities of $\Sigma(X)$ if $L$ separates 4-schemes (see \cite[Lemma 1.1]{Ullery.16}), as follows. 
Set
\[
B^2:=\mathbb{P} (E_{2,L})
\]

\noindent with associated morphism $\pi_2:B^2 \lra X^{[2]}$. Since $L$ is very ample, we have that
$E_{2,L}$ is globally generated, and therefore $\mathcal{O}_{B^2}(1)$ is globally generated as well. 
Since $H^0(B^2,\mathcal{O}_{B^2}(1))=H^0(X^{[2]}, E_{2,L})=H^0(X,L)$, we obtain a morphism
\[
\alpha_2:B^2\lra \mathbb{P}^r.
\]

\noindent We denote $\sigma_2:=\operatorname{Im}(\alpha_2)$, and we note that $\sigma_2=\Sigma(X)$. 

In the sequel, we show that property $N_{3,p}$ holds via the following cohomological criterion (see \cite[Remark 6.4]{Choi.Lacini.Park.Sheridan.25}). 

\begin{theorem}\label{n3p-criterion}
    Let $X$ be a smooth projective variety and let $L$ be a line bundle that separates $4$-schemes. If
    \[
    H^i(X, \wedge ^j M_{1,L}\otimes L)=H^i(X^{[2]}, \wedge^j M_{2,L}\otimes A_{2,L})=0
    \]

\noindent for all $i\geq 1$ and $0\leq j\leq i+p$ then $\Sigma(X)$ satisfies property $N_{3,p}$. 
\end{theorem}

\subsection{Syzygy bundle and normal bundle}\label{subsection-syzygy-normal}
We have seen that for the syzygy bundle $M_L = M_{1,L}$ we have a short exact sequence
\[
0\lra M_L\lra H^0(X,L)\otimes \mathcal{O}_X\lra L\lra 0. 
\]

\noindent As a consequence, we have $M_L=f^* \Omega_{\mathbb{P}^r}^1(1)$, where $f:X\lra \mathbb{P}^r$ is the closed embedding induced by $L$. Therefore, the cotangent sequence of $X\subseteq \mathbb{P}^r$ may be written as
\[
0\lra N_{X/\mathbb{P}^r}^\vee \otimes L\lra M_L \lra \Omega_X ^1 \otimes L\lra 0.
\]

\noindent At the same time, we have an identification
\[
M_L\cong \operatorname{pr}_{2,*} (\operatorname{pr}_1 ^* L \otimes \mathcal{I}_{\Delta_X}),
\]

\noindent where $\Delta_X\subseteq X\times X$ is the diagonal. Finally, we have an identification
\[
\Omega_X ^1 \cong \operatorname{pr}_{2,*}(\mathcal{I}_{\Delta_X}/ \mathcal{I}_{\Delta_X}^2). 
\]

\noindent Consider the short exact sequence
\[
0\lra \mathcal{I}_{\Delta_X}^2 \lra \mathcal{I}_{\Delta_X}\lra \mathcal{I}_{\Delta_X}/ \mathcal{I}_{\Delta_X}^2 \lra 0.
\]

\noindent Twisting by $\operatorname{pr}_1 ^* L$ and pushing forward via $\operatorname{pr}_2$, we obtain an identification
\[
N_{X/\mathbb{P}^r}^\vee \otimes L \cong \operatorname{pr}_{2,*}(\operatorname{pr}_1 ^* L \otimes \mathcal{I}_{\Delta_X}^2). 
\]

\section{Vanishing}\label{section-vanishing}

In this section we prove our main vanishing results, Theorem \ref{vanishing-schur-normal-bundle} and Theorem \ref{vanishing-tensor-product}. These vanishing results will be crucial in proving Theorem \ref{thm1-intro}.

\subsection{A vanishing theorem for the normal bundle}\label{subsection-normal-bundle-vanishing}
To start with, recall that for a vector bundle $F$ of rank $d$, and natural numbers $a\geq 0$, $b\geq 0$ we have the Schur functor
\[
S_{(a, 1^b)} F=\operatorname{ker}\left( \bigwedge^{b} F\otimes \operatorname{Sym}^a F \lra \bigwedge^{b-1}F\otimes \operatorname{Sym}^{a+1}F\right).
\]

\noindent We have $S_{(a,1^b)}F=0$ if $b\geq d$ or $a=0$ and $b>0$. Particular cases of interest are
\begin{enumerate}
    \item $S_{(a, 1^0)}F\cong \operatorname{Sym}^a F$,
    \item $S_{(1, 1^b)}F = \bigwedge^{b+1} F$, and
    \item $S_{(a, 1^{d-1})} F\cong \operatorname{Sym}^{a-1}F\otimes \operatorname{det}(F)$.
\end{enumerate}

\noindent Quite generally, since we work over the complex numbers, Pieri's rule gives
\[
\bigwedge ^b F \otimes \operatorname{Sym}^a F\cong S_{(a,1^b)}F\oplus S_{(a+1, 1^{b-1})}F.
\]

\noindent We need the following result, along the lines of Manivel's vanishing theorem \cite{Manivel.97}.

\begin{lemma}\label{manivel-vanishing}
        Let $X$ be a smooth projective variety, let $F$ be a nef vector bundle of rank $d$, and let $M$ be a big and nef line bundle. Let $a\geq 0$ and $0\leq b\leq d-2$ be two non-negative integers. Then we have
    \[
    H^i(X, \omega_X\otimes \operatorname{Sym}^a F\otimes F^{\otimes b} \otimes \operatorname{det}(F)^{b+1}\otimes M)=0
    \]

\noindent for all $i\geq 1$. (Here we use the convention that $F^{\otimes 0}=\operatorname{Sym}^0 F=\mathcal{O}_X$.)
\end{lemma}
\begin{proof}
Set 
    \[
Y:=\mathbb{P}(F)\times _X \cdots \times_X \mathbb{P}(F),
\]

\noindent where the product is taken $b+1$ times, and let $\pi:Y\lra X$ be the structure morphism. Consider the line bundle
\[
N:=\pi^*M \otimes (\mathcal{O}_Y(a+d)\boxtimes \mathcal{O}_Y(1+d)^{\boxtimes b}). 
\]

\noindent We have that $\mathcal{O}_Y(1)$ is nef since $F$ is nef. In particular, we have that $N$ is nef. Since $\mathcal{O}_Y(1)$ is relatively ample over $X$ and nef on $Y$, and $M$ is big and nef on $X$, we have that
the top intersection product of $N$ is positive. Hence $N$ is big and nef. Therefore, by the Kawamata--Viehweg vanishing theorem, we get
\[
H^i(Y, \omega_Y\otimes N)=0
\]

\noindent for all $i\geq 1$. Recall that
\[
\omega_Y \cong \pi^* (\omega_X \otimes \operatorname{det}(F)^{b+1})\otimes \mathcal{O}_Y(-d)^{\boxtimes b+1}. 
\]

\noindent It is immediate to see that
\[
R^i\pi_*(\omega_Y\otimes N)=
\begin{cases}
    0 & \text{if }i>0, \text{ and}\\
    \omega_X\otimes \operatorname{Sym}^aF \otimes F^{\otimes b}\otimes \operatorname{det}(F)^{b+1}\otimes M & \text{if }i=0.
\end{cases}
\]

\noindent Therefore we are done.
\end{proof}

We apply Lemma \ref{manivel-vanishing} to obtain the following result. 

\begin{theorem}\label{vanishing-schur}
    Let $X$ be a smooth projective variety, let $F$ be a nef vector bundle of rank $d$ and let $M$ be a big and nef line bundle. Let $a\geq 0$ and $b\geq 0$ be two non-negative integers. Set $s:=\operatorname{min}\{d-1,b+1\}$. Then we have
    \[
    H^i(X, \omega_X\otimes S_{(a, 1^b)}F\otimes \operatorname{det}(F)^s\otimes M)=0
    \]

\noindent for all $i\geq 1$. 
\end{theorem}
\begin{proof}
If $a=0$ and $b>0$ then $S_{(a,1^b)}F=0$. If $a=0$ and $b=0$, then $S_{(a,1^b)}F=\mathcal{O}_X$, and we are done by the Kawamata--Viehweg vanishing theorem. So we assume $a\geq 1$.

If $b\geq d$ then $S_{(a, 1^b)}F=0$, hence we may assume $b\leq d-1$. If $d=1$ then $b=s=0$.
Since $F$ is a line bundle and $S_{(a,1^b)}F=F^a$, we may conclude by the Kawamata--Viehweg vanishing theorem.
We assume therefore that $d\ge 2$.
Suppose now that $0\leq b\leq d-2$, so that $s=b+1$. 
Noting that $S_{(a, 1^{b})}F$ is a direct summand of $\operatorname{Sym}^{a} F\otimes F^{\otimes b}$,
 we may conclude by Lemma \ref{manivel-vanishing}.
If $b=d-1$ instead, then
\[
S_{(a,1^b)}F\cong \operatorname{Sym}^{a-1}F\otimes \operatorname{det}F.
\]

\noindent In this case $s=d-1$ and we have that
\[
H^i(X, \omega_X \otimes \operatorname{Sym}^{a-1}F\otimes \operatorname{det}(F)^{d}\otimes M)\cong
H^i(X, \omega_X \otimes \operatorname{Sym}^{a-1}F\otimes \operatorname{det}(F)\otimes M')
\]

\noindent where $M':=M\otimes \operatorname{det}(F)^{d-1}$ is big and nef. Hence we may conclude again by Lemma \ref{manivel-vanishing}.
\end{proof}

Next, we want to extend Theorem \ref{vanishing-schur} to a statement including $N_{X/\mathbb{P}^r}^\vee$ for an appropriate embedding $X\subseteq \mathbb{P}^r$.
Consider the blow up map
\[
\operatorname{bl}:X(t)\lra  X^t\times X
\]

\noindent along $\Delta_{X^t\times X}^{1,t+1}\cup \cdots \cup \Delta_{X^t\times X}^{t,t+1}$, where we denote by $\Delta_X^{i,j}$ the diagonal given by $x_i=x_j$. Let $E^{1,t+1}, \cdots, E^{t,t+1}$ be the corresponding irreducible exceptional divisors and set 
\[
E:=E^{1,t+1}+\cdots + E^{t,t+1}.
\]

\noindent We have
\[
\omega_{X(t)}=\operatorname{bl}^*(\omega_X^{\boxtimes t+1})((n-1)E).
\]

\begin{lemma}\label{vanishing-schur-and-diagonal-1}
    Let $X$ be a smooth projective variety of dimension $n$, and let $A$ be a very ample line bundle. Let $F$ be a nef vector bundle of rank $d$ on $X$, and let $M$ be a big and nef line bundle on $X(t)$. Let $a\geq 0$ and $b\geq 0$ be non-negative integers. Set $s:=\operatorname{min}\{d-1,b+1\}$. Then 
    \[
    H^i\left(X(t), \omega_{X(t)}(-(n+1)E)\otimes \operatorname{bl}^*(((A^{n+1})^{\boxtimes t}\boxtimes A^{(n+1)t})\otimes \operatorname{pr}_{t+1}^*(S_{(a, 1^b)} F\otimes \operatorname{det}(F)^s))\otimes M\right)=0
    \]

\noindent for all $i\geq1$. In particular, we have
\[
R^i \operatorname{bl}_* \omega_{X(t)}(-(n+1)E)=0
\]

\noindent for all $i\geq 1$.
\end{lemma}
\begin{proof}
    We have that
    \[
    \operatorname{bl}^*((A^{n+1})^{\boxtimes t}\boxtimes A^{(n+1)t})(-(n+1)E)
    \]

\noindent is basepoint free. Hence the result follows from Theorem \ref{vanishing-schur}.
\end{proof}

\begin{corollary}\label{vanishing-schur-and-diagonal-2}
    Let $X$ be a smooth projective variety of dimension $n$, and let $A$ be a very ample line bundle. Let $F$ be a nef vector bundle of rank $d$ on $X$, and let $M$ be a big and nef line bundle on $X^t\times X$. Let $a\geq 0$ and $b\geq0$ be non-negative integers. Set $s:=\operatorname{min}\{d-1,b+1\}$ and 
    \[
    \mathcal{J}:=\mathcal{I}_{\Delta_{X^t\times X}^{1,t+1}}^2\otimes \cdots \otimes \mathcal{I}_{\Delta_{X^t\times X}^{t,t+1}}^2.
    \]
    
    \noindent Then 
    \[
    H^i\left(X^t\times X, \mathcal{J}\otimes  ((\omega_X\otimes A^{(n+1)})^{\boxtimes t} \boxtimes (\omega_X\otimes A^{(n+1)t}))\otimes \operatorname{pr}_{t+1}^*(S_{(a, 1^b)} F\otimes \operatorname{det}(F)^s)\otimes M\right)=0
    \]

\noindent for all $i\geq 1$. In particular, we have
\[
R^i \operatorname{pr}_{t+1,*}  (\mathcal{J}\otimes ((\omega_X\otimes A^{n+1})^{\boxtimes t}\boxtimes \mathcal{O}_X)\otimes M)=0
\]

\noindent for all $i\geq 1$. 
\end{corollary}
\begin{proof}
    Immediate from Lemma \ref{vanishing-schur-and-diagonal-1}. 
\end{proof}

We are now in a position to prove the first main vanishing theorem.

\begin{theorem}\label{vanishing-schur-normal-bundle}
    Let $X$ be a smooth projective variety of dimension $n$, and let $A$ be a very ample line bundle. Let $F$ be a nef vector bundle of rank $d$, and let $M$ be a big and nef line bundle. Let $a\geq 0$,  $b\geq0$ and $t\geq 0$ be non-negative integers. Set $s:=\operatorname{min}\{d-1,b+1\}$. 
    Let $N$ be a big and nef line bundle, set $L:=\omega_X\otimes A^{n+1}\otimes N$ and let $X\subseteq \mathbb{P}^r$ be the embedding induced by $L$. 
    Then
    \[
    H^i(X, \omega_X\otimes (N_{X/\mathbb{P}^r}^\vee\otimes L)^{\otimes t}\otimes A^{(n+1)t}\otimes S_{(a, 1^b)} F\otimes \operatorname{det}(F)^s\otimes M)=0
    \]

    \noindent for all $i\geq 1$. 
\end{theorem}
\begin{proof}
We have seen in Subsection \ref{subsection-syzygy-normal} that there is an identification
\[
N_{X/\mathbb{P}^r}^\vee \otimes L\cong \operatorname{pr}_{2,*}(\mathcal{I}_{\Delta_X}^2 \otimes \operatorname{pr}_1^* L). 
\]

\noindent As a consequence, we obtain
\[
\left(N_{X/\mathbb{P}^r}^\vee \otimes L \right)^{\otimes t}\cong \operatorname{pr}_{t+1,*}(\mathcal{I}_{\Delta_X^{1,t+1}}^2 \otimes \cdots 
\otimes \mathcal{I}_{\Delta_X^{t,t+1}}^2 \otimes (L ^{\boxtimes t}\boxtimes \mathcal{O}_X)).
\]

\noindent Therefore we conclude by Corollary \ref{vanishing-schur-and-diagonal-2}.
\end{proof}

\begin{corollary}\label{specialized-vanishing}
    Let $X$ be a smooth projective variety of dimension $n$, and let $A$ be a very ample line bundle. Let $M$ be a big and nef line bundle. Let $a\geq 0$, $b\geq 0$, and $t\geq 0$ be non-negative integers. Set $s:=\operatorname{min}\{n-1,b+1\}$. 
    Let $N$ be a big and nef line bundle, set $L:=\omega_X\otimes A^{n+1}\otimes N$, and let $X\subseteq \mathbb{P}^r$ be the embedding induced by $L$. 
    Then
    \[
    H^i\left(X, \omega_X^{s+1}\otimes (N_{X/\mathbb{P}^r}^\vee\otimes L)^{\otimes t}\otimes A^{(n+1)t+2(a+b+ns)}\otimes S_{(a, 1^b)} \Omega_X^1\otimes M\right)=0
    \]

    \noindent for all $i\geq 1$. 
\end{corollary}
\begin{proof}
    Apply Theorem \ref{vanishing-schur-normal-bundle} to the globally generated vector bundle $F:=\Omega_X^1 \otimes A^2$.
    Note that
    \[
    \operatorname{det}(F)=\omega_X\otimes A^{2n}
    \]

    \noindent and 
    \[
    S_{(a, 1^b)}(\Omega_X^1 \otimes A^2)=S_{(a, 1^b)}\Omega_X ^1 \otimes A^{2(a+b)}.
    \]
\end{proof}

\subsection{A vanishing theorem on the Hilbert scheme}
Let $X$ be a smooth projective variety and let $L$ be a very ample line bundle. Recall that
\[
M_{2,L}:=\operatorname{pr}_{2,*}(\operatorname{pr}_1^* L\otimes \mathcal{I}_{\mathcal{Z}_2})
\]

\noindent is a vector bundle on $X^{[2]}$. Recall also that $\tau:X^{[2,3]}\lra X^{[2]}$ is flat. Set
\[
Y:=X^{[2,3]}\times_{X^{[2]}}\times \cdots \times_{X^{[2]}}X^{[2,3]},
\]

\noindent where the product is taken $t$ times. We have that $Y$ has canonical local complete intersection singularities by the same argument as in \cite[Lemma 4.20]{Choi.Lacini.Park.Sheridan.25}. Furthermore $Y$ is obtained
as the blow up
\[
\operatorname{bl}:Y\lra X^t \times X^{[2]}
\]

\noindent along $\mathcal{Z}_2^{1,t+1}\cup \cdots \cup \mathcal{Z}_2 ^{t,t+1}$ with exceptional divisors
$F_2 ^{1,t+1}, \cdots, F_2 ^{t,t+1}$. Set
\[
G:=\sum_{j=1}^t \operatorname{pr}_j^* F_2.
\]

\noindent We have
\[
\omega_Y=\operatorname{bl}^*(\omega_X ^{\boxtimes t}\boxtimes \omega_{X^{[2]}})((n-1)G). 
\]

\noindent Let $A$ be a very ample line bundle on $X$. Note that $A^2$ separates $3$-schemes, hence by the discussion in Subsection \ref{subsection-hilbert-scheme} (see also \cite[Lemma 2.4]{Choi.Lacini.Park.Sheridan.25}) we have that
\[
\operatorname{res}^*A^2\otimes \tau^* N_{2,A^2}(-F_2)=\rho^* N_{3,A^2}
\]

\noindent is globally generated. Therefore
\[
\operatorname{bl}^*((A^2)^{\boxtimes t}\boxtimes N_{2,A^2}^{t})(-G)
\]

\noindent is globally generated. Let $M$ be a big and nef line bundle on $X^t\times X^{[2]}$. By the Kawamata--Viehweg vanishing theorem, we have
\[
H^i\left(Y, \omega_Y\otimes \operatorname{bl}^*((A^{2n})^{\boxtimes t}\boxtimes N_{2,A^2}^{tn}\otimes M)(-nG)\right)=0
\]

\noindent for all $i\geq 1$. As a consequence we have
\[
H^i\left(X^t\times X^{[2]}, (\omega_X\otimes A^{2n})^{\boxtimes t}\boxtimes (\omega_{X^{[2]}}\otimes N_{2,A^2}^{tn})\otimes M\otimes \mathcal{I}_{\mathcal{Z}_2 ^{1,t+1}}\otimes \cdots \otimes \mathcal{I}_{\mathcal{Z}_2 ^{t,t+1}}\right)=0
\]

\noindent for all $i\geq 1$. Proceeding as in Subsection \ref{subsection-normal-bundle-vanishing}, we obtain our second main vanishing theorem. 

\begin{theorem}\label{vanishing-tensor-product}
    Let $X$ be a smooth projective variety of dimension $n$, and let $A$ be a very ample line bundle. Let $N$ be a big and nef line bundle on $X$ and set $L:=\omega_X\otimes A^{2n} \otimes N$. Let $M$ be a big and nef line bundle on $X^{[2]}$. Let $t\geq 0$ be a non-negative integer. Then 
    \[
    H^i(X^{[2]}, \omega_{X^{[2]}}\otimes M_{2,L}^{\otimes t}\otimes N_{2,A^2}^{tn}\otimes M)=0
    \]

\noindent for all $i\geq 1$. 
\end{theorem}

\section{Filtrations}\label{section-filtrations}
Here we use Theorem \ref{vanishing-schur-normal-bundle} and Theorem \ref{vanishing-tensor-product} from Section \ref{section-vanishing} to obtain a cohomology vanishing on vector bundles naturally arising on $\mathbb{P}(\Omega_X^1)$. 

Let $L$ be a very ample line bundle embedding a smooth projective variety $X$ into $\mathbb{P}^r$. Recall from Subsection \ref{subsection-syzygy-normal} that we have the cotangent sequence

\[
0\lra N_{X/\mathbb{P}^r}^\vee \otimes L \lra M_L \lra \Omega^1 _X \otimes L \lra 0.
\]

\noindent Set $Y:=\mathbb{P}_X(\Omega_X ^1)$ with structure morphism $\pi:Y\lra X$. Consider the relative Euler sequence
\[
0\lra \Omega^1 _{Y/X}(1)\lra \pi^* \Omega^1 _X\lra \mathcal{O}_Y(1)\lra 0.
\]

\noindent These two sequences may be combined into the diagram
\begin{center}
\begin{tikzcd}
& & 0 & 0& \\
& & \pi^* L(1)\ar[r,equals]\ar[u] & \pi^*L (1) \ar[u]& \\
   0\ar[r] &\pi^*(N_{X/\mathbb P ^r} ^\vee \otimes L)\ar[r]& \pi^* M_L \ar[r]\ar[u]& \pi^*(\Omega^1 _X \otimes L)\ar[r]\ar[u]&0 \\
   0\ar[r]& \pi^*(N_{X/\mathbb P ^r} ^\vee \otimes L)\ar[r]\ar[u,equals]& K\ar[r]\ar[u]& \Omega^1 _{Y/X}\otimes \pi^*L(1)\ar[u] \ar[r]& 0\\
   & & 0 \ar[u]& 0\ar[u] & 
\end{tikzcd}.
\end{center}

\noindent We have a filtration
\[
\bigwedge^j K = F^0\supseteq F^1 \supseteq \cdots \supseteq F^j \supseteq F^{j+1}=0
\]

\noindent with 
\begin{align*}
F^t/F^{t+1}&\cong \left(\pi^* \bigwedge^t (N_{X/\mathbb{P}^r}^\vee \otimes L)\right)\otimes \bigwedge^{j-t} \left(\Omega^1 _{Y/X}\otimes \pi^* L(1)\right)\\
&\cong\left(\pi^* \bigwedge^t (N_{X/\mathbb{P}^r}^\vee \otimes L)\right)\otimes\Omega_{Y/X}^{j-t} \otimes \pi^* L^{j-t}(j-t)\\
&\cong \pi^* \left( L^j\otimes \bigwedge^t N_{X/\mathbb{P}^r}^\vee \right)\otimes \Omega_{Y/X}^{j-t}(j-t)
\end{align*}

\noindent for $0\leq t\leq j$. Therefore, we have
\[
R^i \pi_* \left( \wedge ^j K\otimes \pi^* L^k (\ell)\right)=0
\]

\noindent for all $i\geq 1$, $j\geq 1$, $k\geq 0$ and $\ell\geq 0$. Moreover, the vector bundle
\[
\pi_* \left( \wedge ^j K\otimes \pi^* L^k (\ell)\right)
\]

\noindent has a filtration $G^0\supseteq G^1 \supseteq \cdots \supseteq G^j \supseteq G^{j+1}$ with 
\[
G^t/G^{t+1}\cong L^{j+k}\otimes \bigwedge^t N_{X/\mathbb{P}^r}^\vee \otimes S_{(\ell, 1^{j-t})} \Omega_X ^1,
\]

\noindent where we recall
\[
S_{(\ell, 1^{j-t})} \Omega_X ^1=\operatorname{ker}\left( \bigwedge^{j-t} \Omega_X ^1\otimes \operatorname{Sym}^\ell \Omega_X^1 \lra \bigwedge^{j-t-1}\Omega_X ^1 \otimes \operatorname{Sym}^{\ell+1}\Omega_X^1\right).
\]

\begin{theorem}\label{vanishing-filtration}
    Let $X$ be a smooth projective variety of dimension $n$, and let $A$ be a very ample line bundle. Let $j\geq 1$, $\ell\geq 1$
    and $k\geq 2$ be positive integers. Set $s:=\operatorname{min}\{n-1, j+1\}$.
    Define $Y:=\mathbb{P}_X(\Omega_X^1)$ with structure morphism $\pi:Y\lra X$. 
    Let $N$ be a big and nef line bundle and set $L:=\omega_X\otimes A^{m}\otimes N$. If 
    \[
    m\geq (n+2)j+\ell+ns
    \]
    
    \noindent then $H^i(Y, \wedge^j K \otimes \pi^* L^k(\ell))=0$ for all $i\geq 1$. 
\end{theorem}
\begin{proof}
    By the above discussion, it suffices to show 
    \[
    H^i\left(X, L^{j+k}\otimes \bigwedge^t N_{X/\mathbb{P}^r}^\vee \otimes S_{(\ell, 1^{j-t})} \Omega_X ^1\right)=0
    \]

\noindent for all $i\geq 1$. 
Set $u:=\operatorname{min}\{n-1,j-t+1\}$ and note $u\leq s$.
Corollary \ref{specialized-vanishing} gives the vanishing
    \[
    H^i\left(X, \omega_X^{u+1}\otimes (N_{X/\mathbb{P}^r}^\vee\otimes L)^{\otimes t}\otimes A^{(n+1)j+2(\ell+j+ns)}\otimes S_{(\ell, 1^{j-t})} \Omega_X^1\otimes M'\right)=0
    \]

\noindent for all big and nef divisors $M'$ and all $i\geq 1$. To match the two vanishings, define

\begin{align*}
M'&:=L^{j-t+k}\otimes \omega_X ^{-u-1}\otimes A^{-(n+1)j-2(\ell+j+ns)}\\
&=(\omega_X\otimes A^{n+1})^c \otimes A^d\otimes L^{k-2}\otimes N^{j-t+2},
\end{align*}

\noindent where $c:=j-t+1-u\geq 0$, and 
\[
d:=m(j-t+k)-(n+1)j-2(\ell+j+ns)-c(n+1)-m(k-2).
\]

\noindent For $m\geq (n+1)j+(\ell+j+ns)$ we have $d\geq 0$, so $M'$ is big and nef, and we are done.
\end{proof}

\section{Syzygies}
In this section we combine all previous vanishing theorems to prove our main result, Theorem \ref{thm1-intro}. 
To start with, consider the short exact sequence
\[
0\lra \rho^* M_{2,L}\lra \tau^* M_{1,L}\lra \operatorname{res}^*L(-F_1)\lra 0.
\]

\noindent on $X^{[1,2]}$. Recall that
\[
\operatorname{res}\times \tau:X^{[1,2]}\lra X\times X
\]

\noindent is the blow up map along the diagonal $\Delta_X$, with exceptional divisor $F_1\cong \mathbb{P}(\Omega_X^1)$. Restricting to $F_1$, we obtain
\[
0\lra \rho^* M_{2,L}|_{F_1}\lra \pi^* M_L\lra \pi^* L(1)\lra 0.
\]

\noindent Hence, we have the isomorphism
\[
\rho^* M_{2,L}|_{F_1}\cong K, 
\]

\noindent where $K$ is the vector bundle defined in Section \ref{section-filtrations}. Set
\[
Q:=\bigwedge^j M_{2,L}\otimes A_{2,L}.
\]

\noindent We have a short exact sequence
\[
0\lra \rho^*Q(-(a+1)F_1)\lra \rho^*Q(-aF_1)\lra
\rho^*Q(-aF_1)|_{F_1}\lra 0. 
\]

\noindent for any $a\geq 0$. Note the isomorphism
\[
\rho^*Q(-aF_1)|_{F_1}\cong \wedge^j K\otimes \pi^* L^2 (a+2).
\]

\noindent Therefore, by Theorem \ref{vanishing-filtration}, we have
\[
H^i(F_1, \rho^* Q(-aF_1)|_{F_1})=0
\]

\noindent for $a\geq 0$, $i\geq 1$ and $L$ of the form
\[
L:=\omega_X\otimes A^{m}\otimes N,
\]

\noindent where $N$ is a big and nef line bundle on $X$, and $m\geq (n+2)j +a+2+ns$. Hence in this range we get a surjection
\[
H^i(X^{[1,2]}, \rho^*Q(-(a+1)F_1))\lra H^i(X^{[1,2]}, \rho^* Q(-aF_1)).
\]

\noindent Recall that $\rho:X^{[1,2]}\lra X^{[2]}$ is a $2:1$ finite flat map. Note that $E_{2, \mathcal{O}_X}$ is a rank 2 vector bundle on $X^{[2]}$ containing $\mathcal{O}_{X^{[2]}}$ as a direct summand. Since $\operatorname{det}(E_{2,\mathcal{O}_X})=\delta_2^{-1}$, we have $E_{2,\mathcal{O}_X}\cong \mathcal{O}_{X^{[2]}}\oplus \delta_2 ^{-1}$. 
From the discussion in Subsection \ref{subsection-hilbert-scheme}, we also have 
\[
\rho^* \delta_2 \cong \mathcal{O}_{X^{[1,2]}}(F_1). 
\]

\noindent Hence
\[
\rho_* \mathcal{O}_{X^{[1,2]}}(-aF_1)\cong \rho_*\rho^* \delta_2 ^{-a}\cong \delta_2^{-a}\oplus \delta_2 ^{-a-1}. 
\]

\noindent Therefore, we get a surjection
\[
H^i(X^{[2]}, Q\otimes \delta_2 ^{-a-1})\oplus H^i(X^{[2]}, Q\otimes \delta_2 ^{-a-2})\lra
H^i(X^{[2]}, Q\otimes \delta_2 ^{-a})\oplus H^i(X^{[2]}, Q\otimes \delta_2 ^{-a-1}).
\]

\noindent By Theorem \ref{vanishing-tensor-product}, we have that 
\[
H^i(X^{[2]}, Q\otimes \delta_2 ^{-a-1})\oplus H^i(X^{[2]}, Q\otimes \delta_2 ^{-a-2})=0
\]

\noindent provided that $a+3\geq jn$. In order to satisfy property $N_{3,p}$, it suffices to check the
vanishing for all $0\leq j\leq i+p\leq 2n+p$ by Theorem \ref{n3p-criterion}. Hence it is sufficient to have 
\begin{align*}
m&\geq (n+2)(2n+p)+(2n+p)n+n(n-1)\\
&=(2n+p)(2n+2)+n(n-1).
\end{align*}

\vspace{30pt}


\begin{thebibliography}{EGHAP05}
\setlength{\parskip}{3pt}


\bibitem[BL25]{Bangere.Lacini.25}
Purnaprajna Bangere and Justin Lacini,
\emph{Syzygies of adjoint linear series on projective varieties}, Duke Math. J. \textbf{174} (2025), 473--499.

\bibitem[CLPS25]{Choi.Lacini.Park.Sheridan.25}
Doyoung Choi, Justin Lacini, Jinhyung Park, John Sheridan, \emph{Singularities and syzygies of secant varieties of smooth projective varieties}, arXiv:2502.19703v1 (2025), 77 pages.

\bibitem[CP26]{Choi.Park.26}
Doyoung Choi, Jinhyung Park, \emph{Effective results on projective normality of the first and second secant varieties}, arXiv:2606.30002v1 (2026), 19 pages.

\bibitem[CS18]{Chou.Song.18}
Chih-Chi Chou and Lei Song, \textit{Singularities of secant varieties}, Int. Math. Res. Not. \textbf{9} (2018), 2844--2865.

\bibitem[EL93]{Ein.Lazarsfeld.93}
Lawrence Ein and Robert Lazarsfeld,
\emph{Syzygies and Koszul cohomology of smooth projective varieties of arbitrary dimension},
Invent. \!Math. \textbf{111} (1993), 51--67.

\bibitem[Man97]{Manivel.97}
Laurent Manivel, \emph{Vanishing theorems for ample vector bundles}, Invent. Math. \textbf{127} (1997), 401--416.

\bibitem[Ull16]{Ullery.16}
Brooke Ullery, \emph{On the normality of secant varieties},
Adv. \!Math. \textbf{288} (2016), 631--647.


\end{thebibliography}

\end{document}